\documentclass [12pt,leqno,fleqn]{amsart}

\usepackage{amsmath,amstext,amsthm,amsxtra,mathtools,amssymb}

\usepackage[colorlinks=true,citecolor=blue,pagebackref,citecolor=magenta, urlcolor=cyan , hypertexnames=true]{hyperref}

\usepackage[normalem]{ulem}
\usepackage{cancel}
\usepackage{txfonts} 
\usepackage[T1]{fontenc}
\usepackage{lmodern}
\usepackage{marginnote}

\usepackage{enumitem}

\usepackage{tikz} 

\usepackage[backrefs,msc-links,nobysame,initials]{amsrefs}

\allowdisplaybreaks

\usepackage{todonotes}
\makeatletter

\def\R{\mathbb R}
\def\N{\mathbb N}

\def\Z{\mathbb Z}

\def\M{\mathsf M}      
\def\Ms{\mathsf M _{\mathsf S} ^*} 
\def\MHL{ \mathsf M _{\mathsf {HL}}}  
\def\MHLC{ \mathsf M _{\mathsf {HL, c}}} 
\def\Mstr{\mathsf M _{\mathsf S}}  
\def\MHLCE{\mathsf M^\ast_{\mathsf {HL,c}}} 
\def\MHLE{\mathsf M^\ast_{\mathsf {HL}}} 
\def\MHLD{\tilde{\mathsf M}_{\mathsf {HL}}}
\def\MHLCD{\tilde{\mathsf M}_{\mathsf {HL,c}}}
\def\Cs{\mathsf C^* _{\mathsf S}}  
\def\CSD{\tilde{\mathsf C}_{\mathsf S}} 

\def\Cstr{\mathsf C _{\mathsf{S}}} 
\def\CHL{\mathsf C_{\mathsf{HL}}} 
\def\CHLE{\mathsf C_{\mathsf{HL} } ^\ast} 
\def\CHLC{\mathsf C_{\mathsf{HL, c}}}
\def\CHLCE{\mathsf C_{\mathsf{HL, c} } ^\ast}
\def\CHLCD{\tilde{\mathsf C}_{\mathsf{HL, c}}}
\def\CHLD{\tilde{\mathsf C}_{\mathsf{HL}}} 
\def\DSM{\tilde{\mathsf M}_{\mathsf S}} 
\def\CEB{\mathsf C^* _{\mathcal{B}}(\alpha)} 
\def\MEB{\mathsf M^\ast _{\mathcal{B}}}
\def\MB{\mathsf M_{\mathcal{B}}}
\def\MDB{\tilde{\mathsf M}_\mathcal{B}} 
\def\CB{\mathsf C_{\mathcal{B}}(\alpha)}
\def\CDB{ \tilde{\mathsf C}_{\mathcal{B}}(\alpha)}
\def\norm#1.#2.{\lVert#1\rVert_{#2}}
\def\Norm#1.#2.{\bigl\lVert#1\bigr\rVert_{#2}}
\def\NOrm#1.#2.{\Bigl\lVert#1\Bigr\rVert_{#2}}
\def\NORm#1.#2.{\biggl\lVert#1\biggr\rVert_{#2}}
\def\NORM#1.#2.{\Biggl\lVert#1\Biggr\rVert_{#2}}

\def\ip#1,#2,{\langle #1,#2\rangle}
\def\Ip#1,#2,{\bigl\langle#1,#2\bigr\rangle}
\def\IP#1,#2,{\Bigl\langle#1,#2\Bigr\rangle}

\def\Abs#1{\bigl\lvert#1\bigr\rvert}
\def\ABs#1{\Bigl\lvert#1\Bigr\rvert}

\theoremstyle{plain}
\newtheorem{theorem}[equation]{Theorem}

\theoremstyle{definition}
\newtheorem{definition}[equation]{Definition}
\newtheorem{remark}[equation]{Remark}

\newtheorem{problem}[equation]{Problem}

\numberwithin{equation}{section}
\title[Solyanik estimates in ergodic theory]{Solyanik estimates in ergodic theory}

\author{Paul Hagelstein}
\address{P.H.: Department of Mathematics, Baylor University, Waco, Texas 76798}
\email{\href{mailto:paul_hagelstein@baylor.edu}{paul\!\hspace{.018in}\_\,hagelstein@baylor.edu}}
\thanks{P. H. is partially supported by a grant from the Simons Foundation (\#208831 to Paul Hagelstein).}

\author{Ioannis Parissis}
\address{I.P.: Departamento de Matem\'aticas, Universidad del Pais Vasco, Aptdo. 644, 48080
Bilbao, Spain and Ikerbasque, Basque Foundation for Science, Bilbao, Spain}
\email{\href{mailto:ioannis.parissis@ehu.es}{ioannis.parissis@ehu.es}}
\thanks{I. P. is supported by IKERBASQUE.}


\begin{document}
	
\begin{abstract}
Let $U_1, \ldots, U_n$ be a collection of commuting measure preserving transformations on a probability space $(\Omega, \Sigma, \mu)$.   Associated with these measure preserving transformations is the ergodic strong maximal operator $\Ms$ given by
\[
\Ms  f(\omega) \coloneqq \sup_{0 \in R  \subset  \mathbb{R}^n}\frac{1}{\#(R \cap \mathbb{Z}^n)}\sum_{(j_1, \ldots, j_n) \in R\cap \mathbb{Z}^n}\big|f(U_1^{j_1}\cdots U_n^{j_n}\omega)\big|,
\]
where the supremum is taken over all open rectangles in $\mathbb{R}^n$ containing the origin whose sides are parallel to the coordinate axes. For $0 < \alpha < 1$ we define the \emph{sharp Tauberian constant of $\Ms$ with respect to $\alpha$} by
\[
\Cs (\alpha) \coloneqq \sup_{\substack{E \subset \Omega \\ \mu(E) > 0}}\frac{1}{\mu(E)}\mu(\{\omega \in \Omega : \Ms \chi_E (\omega) > \alpha\}).
\]
Motivated by previous work of A. A. Solyanik and the authors regarding Solyanik estimates for the geometric strong maximal operator in harmonic analysis, we show that the Solyanik estimate
\[
\lim_{\alpha \rightarrow 1}\Cs(\alpha) = 1
\]
holds, and that in particular we have
\[\Cs(\alpha) - 1 \lesssim_n (1 - \frac{1}{\alpha})^{1/n}\]
provided that $\alpha$ is sufficiently close to $1$. Solyanik estimates for centered and uncentered ergodic Hardy-Littlewood maximal operators associated with $U_1, \ldots, U_n$ are shown to hold as well. Further directions for research in the field of ergodic Solyanik estimates are also discussed.
\end{abstract}

\maketitle

\section{Introduction}\label{intro}
This paper is intended to be an introduction of the topic of \emph{Solyanik estimates} to the field of ergodic theory. Solyanik estimates first emerged in the field of harmonic analysis in the mid 1990's with the work of A. A. Solyanik \cite{Solyanik} regarding fine properties of the restricted weak type distribution functions of the Hardy-Littlewood and strong maximal functions.  Recall that the uncentered Hardy-Littlewood maximal operator $\MHL$ is defined on functions $f \in L^{1}_{loc}(\mathbb{R}^n)$ by
\[
\MHL  f(x) \coloneqq \sup_{x \in B \subset \R^n}\frac{1}{|B|}\int_B |f|,
\]
where the supremum is taken over the set of all balls $B$ in $\mathbb{R}^n$ containing $x$.  The closely related \emph{centered} Hardy-Littlewood maximal operator $\MHLC$ is defined by
\[
\MHLC  f(x) \coloneqq \sup_{r>0} \frac{1}{|B(x,r)|}\int_{B(x,r)} |f(y)|dy
\]
where the supremum is taken over the set of all balls $B(x,r)$ in $\mathbb{R}^n$ that are centered at $x$ and $f$ is a locally integrable function on $\R^n$. The strong maximal operator $\Mstr $ is defined on   locally integrable functions on $\mathbb{R}^n$ by
\[
\Mstr f(x) \coloneqq \sup_{x \in R \subset \mathbb{R}^n}\frac{1}{|R|}\int_R |f|,
\]
where the supremum is taken over the set of rectangles in $\mathbb{R}^n$ containing $x$ whose sides are parallel to the coordinate axes.

The Hardy-Littlewood maximal operator is relatively easily seen to satisfy the weak type estimate
\[
|\{x \in \mathbb{R} : \MHL f (x) > \alpha\}| \leq \frac{3^n}{\alpha}\|f\|_{L^1(\R^n)};
\]
the sharp weak type estimate for $\MHL$ acting on functions on $\mathbb{R}$ may be improved, as is shown by Grafakos and Montgomery-Smith in \cite{grafakosmontsmith1997}. From this the restricted weak type estimate
\[
|\{x \in \mathbb{R} : \MHL\chi_E(x) > \alpha\}| \leq \frac{3^n}{\alpha}|E|
\]
immediately follows.  The centered Hardy-Littlewood maximal operator $\MHLC$ satisfies a similar restricted weak type estimate. Now, these estimates hold for $\MHL$ and $\MHLC$ for all $0<\alpha<1$ but it is reasonable to expect that the quantity $3^n/\alpha$ may be replaced by a value arbitrarily close to $1$ provided that we only consider values of $\alpha$ sufficiently near $1$. This expectation is validated by results collectively due to the authors and Solyanik. In order to state these results in a precise manner we introduce the following definitions.

\begin{definition}  Let $\MHLC$ denote the centered Hardy-Littlewood maximal operator  on $\mathbb{R}^n$. The \emph{sharp Tauberian constant} of $\MHLC $ with respect to $\alpha\in(0,1)$ is defined as
\[
\CHLC (\alpha) \coloneqq \sup_{\substack{E \subset \mathbb{R}\\ 0 < |E| < \infty}}\frac{1}{|E|}|\{x \in \mathbb{R} : \MHLC \chi_E(x) > \alpha\}|.
\]
Similarly, let $\MHL$ denote the uncentered Hardy-Littlewood maximal operator  on $\mathbb{R}^n$. The sharp Tauberian constant of $\MHL $ with respect to $\alpha\in(0,1)$ is defined as
\[
\CHL (\alpha) \coloneqq \sup_{\substack{E \subset \mathbb{R}\\ 0 < |E| < \infty}}\frac{1}{|E|}|\{x \in \mathbb{R} : \MHL \chi_E(x) > \alpha\}|.
\]	
Finally, let $\Mstr$ denote the geometric strong maximal operator on $\R^n$. The sharp Tauberian constant of $\Mstr$ with respect to $\alpha$ is defined as
\[
\Cstr (\alpha) \coloneqq \sup_{\substack{E \subset \mathbb{R}^n \\ 0 < |E| < \infty}}\frac{1}{|E|}|\{x \in \mathbb{R}^n : \Mstr\chi_E(x) > \alpha\}|.
\]
\end{definition}

The following theorem provides asymptotic estimates as $\alpha\to 1^-$ for the sharp Tauberian constants of the geometric maximal operators defined above.

\begin{theorem}[Hagelstein, Parissis, \cite{HP}, Solyanik \cite{Solyanik}]\label{solyanikthm} We have the following Solyanik estimates, for the centered, and uncentered sharp Tauberian constants.
\begin{enumerate}[leftmargin=3em,labelindent=-2pt]
\item[(i)] Let $\CHLC (\alpha)$ denote the sharp Tauberian constant of $\MHLC $ with respect to $\alpha$. Then $\lim_{\alpha \rightarrow 1^-}\CHLC (\alpha) = 1$.
In particular we have that
\[
\CHLC (\alpha) - 1 \lesssim_n  \frac{1}{\alpha} - 1
\]
for $\alpha$ sufficiently close to $1$.

\item[(ii)] Let $\CHL (\alpha)$ denote the sharp Tauberian constant of $\MHL $ with respect to $\alpha$. Then   $\lim_{\alpha \rightarrow 1^-}\CHL (\alpha) = 1$.  In particular we have that
\[
\CHL (\alpha) - 1 \lesssim_n \left(\frac{1}{\alpha} - 1\right)^{\frac{1}{n+1}}
\]
for $\alpha$ sufficiently close to 1.
\end{enumerate}
\end{theorem}

Analogues of this result also exist for the strong maximal operator. Indeed, we have

\begin{theorem}[Hagelstein, Parissis, \cite{HP}, Solyanik, \cite{Solyanik}]\label{strongsolyanik}
Let $\Cstr (\alpha)$ denote the sharp Tauberian constant of $\Mstr$ with respect to $\alpha$. Then $\lim_{\alpha \rightarrow 1^-}\Cstr (\alpha) = 1$. In particular we have that
\[
\Cstr (\alpha) - 1 \lesssim_n \left(\frac{1}{\alpha} - 1\right)^{\frac{1}{n}}
\]
for $\alpha$ sufficiently close to 1.
\end{theorem}

Now, as is well-known, a close relationship exists between the distribution functions of geometric maximal functions commonly arising in harmonic analysis and their counterparts in ergodic theory. Papers describing this correspondence range from the pioneering work of Calder\'on on transference principles, see~\cite{calderon1968}, to the more recent work of Hagelstein and Stokolos, \cite{hs2012}. Hence it is quite natural to inquire as to whether the above \emph{Solyanik  estimates} have analogues in the ergodic theoretic context. The purpose of this paper is to show that desirable Solyanik estimates do indeed exist for ergodic analogues of the centered and uncentered Hardy-Littlewood maximal operators as well as the strong maximal operator. Our techniques will be rather classical, relying on the above estimates for the geometric strong maximal operator as well as the ideas behind the general transference principles of Calder\'on. We will also indicate intriguing directions for future work regarding Solyanik estimates in ergodic theory.

It is interesting to observe that, although the Solyanik estimate for the Hardy-Littlewood maximal operator was first observed and proven only twenty years ago, the corresponding Solyanik estimate for  one-parameter ergodic maximal operators have been known since the infancy of ergodic theory. In particular, let $T$ be a measure preserving transformation on a probability space $(\Omega, \Sigma, \mu)$.   We may associate to $T$ the maximal operator $T^\ast$ defined by
\[
T^\ast f(\omega) \coloneqq \sup_{N\geq 1}\frac{1}{N}\sum_{j=0}^{N-1}|f(T^j\omega)|.
\]
Then
\[
\mu(\{\omega \in \Omega : T^\ast f(\omega) > \alpha\} \leq \frac{1}{\alpha}\int_{\Omega}|f|\;d\mu\;.
\]
This result goes back to Wiener \cite{Wiener39} and Yosida and Kakutani \cite{YosidaKakutani1939}.  One may consult Petersen \cite{Petersen} for a more recent presentation of a proof of this result. This directly implies that
\[
\mu(\{\omega \in \Omega : T^\ast\chi_E(\omega) > \alpha\}) \leq \frac{1}{\alpha}\mu(E)\;,
\]
and hence
\[
\lim_{\alpha \rightarrow 1^-}\sup_{\substack{E \subset \Omega \\ \mu(E) > 0}}\frac{1}{\mu(E)}\mu(\{\omega \in \Omega : T^\ast \chi_E (\omega) > \alpha\}) = 1\;.
\]
However, this estimate may not be iterated to directly achieve Solyanik estimates for \emph{multiparameter} ergodic maximal operators, in particular for ergodic maximal operators associated to multiple measure preserving transformations. Using additional ideas we will show that Solyanik estimates for multiparameter ergodic maximal operators indeed do hold. For specificity, we will now explicitly define analogues of the centered and uncentered Hardy-Littlewood maximal operators in the ergodic setting and state three theorems regarding the Solyanik estimates associated to these ergodic maximal operators, the proofs of which constitute the following four sections of this paper.

We first introduce appropriate collections of sets in $\R^n$ which we will use in order to define our ergodic maximal operators. These collections will be liberally called  \emph{bases} and will be used throughout the paper.   Being very specific, in this paper a \emph{basis} in $\mathbb{R}^n$ will be a collection of bounded open subsets of $\mathbb{R}^n$ containing the origin.  We will be giving close consideration to three particular bases and accordingly give them special notation.

\begin{definition} We denote by  $\mathcal B_{\mathsf{S}}$ the collection of all open rectangles in $\R^n$ which contain the origin  and have sides parallel to the coordinate axes, by $\mathcal B_{\mathsf{HL}}$ the collection of all open Euclidean balls in $\R^n$ which contain the origin, and by $\mathcal B_{\mathsf {HL,c}}$ the collection of all open Euclidean balls in $\R^n$ which are centered at the origin.
\end{definition}

Let now $\mathcal B$ be a basis in $\R^n$. We will consider three types of maximal operators associated with $\mathcal B$ together with their associated Tauberian constants.

\begin{definition} Let $\mathcal B$ be a basis in $\R^n$.
\begin{enumerate}[leftmargin=3em,labelindent=-2pt]	
\item[(i)] The \emph{geometric maximal operator} $\M_{\mathcal B}$ associated with $\mathcal B$ is defined on $f\in L^1_{loc}(\R^n)$ as
\[
\M_{\mathcal B}f(x)\coloneqq \sup_{R\in\mathcal B} \frac{1}{|R|}\int_R |f(x+y)|dy,\quad x\in \R^n.
\]
\item[(ii)] The \emph{discrete geometric maximal operator} $\tilde {\M}_{\mathcal B}$ associated with $\mathcal B$ is defined on $f \in L^1_{loc}(\Z^n)$ as
\[
\tilde{\M}_{\mathcal B}f(m)\coloneqq \sup_{R \in \mathcal B}\frac{1}{\#(R\cap \Z^n)} \sum_{j=(j_1, \ldots, j_n)\in R\cap \mathbb{Z}^n} |f(m+j)|,\quad m\in \Z^n.
\]
 \item[(iii)] Let $U_1,\ldots,U_n$ be a collection of measure preserving transformations on a probability space $(\Omega,\Sigma,\mu)$. The \emph{ergodic maximal operator} $\M^\ast _{\mathcal B}$ associated with $\mathcal B$ is the maximal operator defined on $f \in L^1(\Omega)$ as
\[
\M ^\ast_{\mathcal{B}} f(\omega) \coloneqq \sup_{R\in \mathcal B} \frac{1}{\#(R\cap \Z^n)} \sum_{j=(j_1, \ldots, j_n)\in R\cap \mathbb{Z}^n} |f(U_1 ^{j_1}U_2 ^{j_2}\cdots U_n ^{j_n}\omega )|,\quad \omega \in \Omega.
\]
\end{enumerate}
\end{definition}
For these geometric, discrete, and ergodic maximal operators there is a natural definition of the corresponding sharp Tauberian constants. We make this precise below.
\begin{definition}
Let $\mathcal{B}$ be a collection of bounded open subsets of $\mathbb{R}^n$ containing the origin and let $0 < \alpha < 1$. The Tauberian constants associated to the maximal operators $\MB$, $\MDB$, and $\MEB$ are defined respectively by
\[
\begin{split}
\CB &= \sup_{\substack{E \subset \mathbb{R}^n \\ 0 < |E| < \infty}}\frac{1}{|E|}\left|\left\{x\in \mathbb{R}^n : \MB\chi_E(x) > \alpha\right\}\right|,
\\
\CDB &= \sup_{\substack{E \subset \mathbb{Z}^n \\ 0 < \#E < \infty}} \frac{1}{\#E}\#\left\{m \in \mathbb{Z}^n : \MDB\chi_E(m) > \alpha\right\}, \quad\text{and}
\\
\CEB &= \sup_{\substack{E \subset \Omega \\ \mu(E)>0}} \frac{1}{\mu(E)}\mu\left(\left\{\omega \in \Omega : \MEB \chi_E(\omega) > \alpha\right\}\right).
\end{split}
\]
\end{definition}

\begin{remark}
As noted before, in this paper will will be primarily interested in the bases  $\mathcal{B}_{\mathsf S}, \mathcal{B}_{\mathsf {HL}}$, and $\mathcal{B}_{\mathsf {HL,c}}$, corresponding to the strong maximal operator, the uncentered Hardy-Littlewood maximal operator, and the centered Hardy-Littlewood maximal operator, respectively. As a convenient shorthand notation, we will denote the ergodic maximal operators $\mathsf{M}^\ast_{\mathcal{B}_{\mathsf S}}$, $\mathsf{M}^\ast_{\mathcal{B}_{\mathsf {HL}}}$, and $\mathsf{M}^\ast_{\mathcal{B}_{\mathsf {HL,c}}}$  respectively by $\Ms$, $\MHLE$, and $\MHLCE$.  We will also denote the discrete maximal operators $\tilde{\mathsf M}_{\mathcal{B}_{\mathsf S}}$, $\tilde{\mathsf M}_{\mathcal{B}_{\mathsf {HL}}}$, and $\tilde{\mathsf M}_{\mathcal{B}_{\mathsf {HL,c}}}$  respectively by $\DSM$, $\MHLD$, and $\MHLCD$.

The same notational conventions will be consistently applied for the Tauberian constants corresponding to the bases $\mathcal B_{\mathsf S}, \mathcal B_{\mathsf{HL}}$, and $\mathcal B_{\mathsf{HL,c}}$ in both the ergodic and discrete contexts.   Thus we will have
\[
\begin{split}
\CSD(\alpha) &\coloneqq  \tilde{\mathsf{C}}_{\mathcal{B}_{\mathsf S}}(\alpha) =\sup_{E \subset \mathbb{Z}^n \atop 0 < \#E < \infty}\frac{1}{\#E}\#\{x \in \mathbb{Z}^n : \tilde{\mathsf{M}}_{\mathsf S}(x) > \alpha\},
	\\
\CHLD(\alpha) &\coloneqq  \tilde{\mathsf{C}}_{\mathcal{B}_{\mathsf {HL}}}(\alpha) = \sup_{E \subset \mathbb{Z}^n \atop 0 < \#E < \infty}\frac{1}{\#E}\#\{x \in \mathbb{Z}^n : \tilde{\mathsf{M}}_{\mathsf {HL}}(x) > \alpha\},
	\\
\CHLCD(\alpha) &\coloneqq  \tilde{\mathsf{C}}_{\mathcal{B}_{\mathsf {HL,c}}}(\alpha) = \sup_{E \subset \mathbb{Z}^n \atop 0 < \#E < \infty}\frac{1}{\#E}\#\{x \in \mathbb{Z}^n : \tilde{\mathsf{M}}_{\mathsf {HL,c}}(x) > \alpha\},
	\\
\Cs (\alpha) &\coloneqq \mathsf{C} ^\ast _{\mathcal B_{\mathsf S}} (\alpha)=  \sup_{\substack{E \subset \Omega \\   \mu(E)>0}}\frac{1}{\mu(E)}\mu(\{\omega \in \Omega : \Ms  \chi_E(\omega) > \alpha\})\;,
\\
\CHLE (\alpha) &\coloneqq \mathsf{C} ^\ast _{\mathcal B_{\mathsf {HL}}}(\alpha)  = \sup_{\substack{E \subset \Omega \\  \mu(E)>0}}\frac{1}{\mu(E)}\mu(\{\omega \in \Omega : \MHLE  \chi_E(\omega) > \alpha\}), \quad\text{and}
\\
\CHLCE (\alpha) &\coloneqq \mathsf{C} ^\ast _{\mathcal B_{\mathsf {HL,c}}}(\alpha) = \sup_{\substack{ E \subset \Omega \\   \mu(E)>0} }\frac{1}{\mu(E)}\mu(\{\omega \in \Omega : \MHLCE \chi_E(\omega) > \alpha\}).
\end{split}
\]
\end{remark}

The following theorems are the main results of this paper and provide the ergodic theoretic  Solyanik estimates for the ergodic strong, centered, and uncentered Hardy-Littlewood maximal operators, respectively.
\begin{theorem}\label{t1}
 Let $\Ms$ denote the ergodic strong maximal operator associated with a collection  $U_1, \ldots, U_n$ of commuting measure preserving transformations on a probability space $(\Omega, \Sigma, \mu)$, and let $\Cs(\alpha)$ be the associated sharp Tauberian constants for $0<\alpha<1$.  Then $\lim_{\alpha \rightarrow 1^-} \Cs(\alpha) = 1 $. In particular we have that
\[
\Cs (\alpha) - 1 \lesssim_n \big(\frac{1}{\alpha} - 1\big)^{1/n}
\]
for $\alpha$ sufficiently close to $1$.
\end{theorem}

\begin{theorem}\label{t2}
 Let $\MHLCE$ denote the ergodic centered Hardy-Littlewood maximal operators associated with a collection  $U_1, \ldots, U_n$  of commuting measure preserving transformations on a probability space $(\Omega, \Sigma, \mu)$, and let $\CHLCE$ be the associated sharp Tauberian constants for $0 < \alpha < 1$.  Then $\lim_{\alpha \rightarrow 1^-} \CHLCE (\alpha) = 1$. In particular we have that
\[
\CHLCE (\alpha) - 1 \lesssim_n \frac{1}{\alpha} - 1
\]
for $\alpha$ sufficiently close to $1$.
\end{theorem}

\begin{theorem}\label{t3}
 Let $\MHLE$ denote the ergodic Hardy-Littlewood maximal operators associated with a collection  $U_1, \ldots, U_n$  of commuting measure preserving transformations on a probability space $(\Omega, \Sigma, \mu)$, and let $\CHLE(\alpha)$ be the associated sharp Tauberian constants for $0 < \alpha < 1$.  Then $\lim_{\alpha \rightarrow 1^-} \CHLE (\alpha) = 1$. In particular we have that
\[
\CHLE  (\alpha) - 1 \lesssim_n \left(\frac{1}{\alpha} - 1\right)^{\frac{1}{n(n+1)}}
\]
for $\alpha$ sufficiently close to $1$.
\end{theorem}

\section{Notation} We use the letters $C,c$ to denote positive numerical constants whose value might change even in the same line of text. We indicate the dependence of some constant $C$ on a parameter $n$ by writing $C_n$. We use the notation $A\lesssim B$ whenever $A\leq C B$. If the implicit constant depends on some parameter $n$ we write $A\lesssim_n B$.
\section{Transference of Solyanik estimates}\label{s.trans}
The purpose of this section is to provide a general transference principle that will enable us to transfer Solyanik type estimates for discrete maximal operators acting on $L^1(\mathbb{Z}^n)$ to their ergodic counterparts.

\begin{theorem}\label{tt}
Let $U_1, \ldots, U_n$ be a collection of commuting measure preserving transformations on a probability space $(\Omega, \Sigma, \mu)$ and $\mathcal{B}$ be a collection of nonempty bounded open subsets of  $\R^n$. Let $\MDB$ and $\M ^\ast_{\mathcal{B}}$ be the discrete and ergodic maximal operators as above, and  $\CDB$ and $\CEB$ their respective sharp Tauberian constants. For every $0 < \alpha < 1$ we have
  \[
  \CEB \leq \CDB.
  \]
\end{theorem}

\begin{proof}
We  proceed by taking advantage of transference principals developed by A. P. Calder\'on in \cite{calderon1968}.
Let $E$ be a measurable subset of $\Omega$ and $T>0$. We associate with $E$ the function $F_{E,T}$ on $\Omega \times \mathbb{Z}^n$ given by
\[
F_{E, T}(\omega, t) \coloneqq \chi_E(U_1^{t_1}\cdots U_n^{t_n}\omega) \chi_{(-T,T)^n}(t),\quad \omega\in\Omega,\quad t=(t_1,\ldots,t_n)\in \mathbb{Z}^n.
\]
Note that for fixed $t\in (-T,T)^n$ we have that the functions $F_{E,T}(\cdot, t)$ and $\chi_E$ are equimeasurable on $\Omega$.

Let $\omega \in \Omega$ be fixed and let $r > 0$. Also let $R \in \mathcal{B}$ such that $R \subset (-r, r)^n$.   If $m=(m_1, \ldots, m_n)\in\mathbb{Z}^n\cap[-T,T]^n$, we can write
\[
\begin{split}
&\frac{1}{\#R}\sum_{j=(j_1, \ldots, j_n)\in R}\chi_E(U_1^{j_1}\cdots U_n^{j_n}\omega)
\\
&\qquad=\frac{1}{\#R}\sum_{j=(j_1, \ldots, j_n)\in R}\chi_E(U_1^{m_1 + j_1}\cdots U_n^{m_n + j_n}U_1^{-m_1}\cdots U_n^{-m_n}\omega)\chi_{(-r-T,r+T)^n}(m+j)
\\
&\qquad= \frac{1}{\#R}\sum_{j=(j_1, \ldots, j_n)\in R} F_{E,r+T}(U_1^{-m_1}\cdots U_n^{-m_n}\omega,  m  + j )
\end{split}
\]
using the hypothesis that $U_1, \ldots, U_n$ commute. Defining the discrete maximal operator $\tilde{\M}_{\mathcal{B}, r}$ by
\[
\tilde{\M}_{\mathcal{B}, r}f(x) \coloneqq \sup_{\substack{ R \in \mathcal{B} \\ R \subset (-r,r)^n}}\frac{1}{\#R}\sum_{(j_1, \ldots, j_n)\in R}|f(x_1 + j_1, \ldots, x_n + j_n)|,\quad x=(x_1,\ldots,x_n)\in\Z^n
\]
and the ergodic maximal operator $\M^\ast_{\mathcal{B}, r}$ by
\[
\M ^\ast_{\mathcal{B}, r}f(\omega) \coloneqq \sup_{\substack{R \in \mathcal{B} \\ R \subset (-r,r)^n}} \frac{1}{\#R}\sum_{(j_1, \ldots, j_n)\in R}|f(U_1^{j_1}\cdots U_n^{j_n}\omega)|,\quad \omega\in\Omega,
\]
we have
\[
\M^\ast_{\mathcal{B},r} \chi_E (\omega) = [\tilde{\M}_{\mathcal{B},r}F_{E, r + T}(U_1^{-m_1}\cdots U_n^{-m_n}\omega, \cdot)](m).
\]
Now, using the fact that $U_1,\ldots,U_n$ are measure preserving together with Fubini's theorem we can conclude that for any $r,T\in\N$ we have
\[
\begin{split}
&\mu(\{\omega \in \Omega : \M^{\ast}_{\mathcal{B},r} \chi_E(\omega) > \alpha\})
\\
& = \frac{1}{(2T + 1)^n}\sum_{m_1 = -T} ^T \cdots \sum_{m_n = -T} ^T \mu(\{\omega \in \Omega :\M^{\ast}_{\mathcal{B},r } \chi_E(U_1^{m_1}\cdots U_n^{m_n}\omega) > \alpha\})
\\
& = \frac{1}{(2T + 1)^n}\sum_{m_1 = -T}^T \cdots \sum_{m_n = -T} ^T \mu(\{\omega \in \Omega : [ \tilde{\M}_{\mathcal{B},r} F_{E,r + T}( \omega, \cdot) ] (m) >\alpha\})
\\
&=\frac{1}{(2T + 1)^n} \int_{\Omega} \#\{ m \in \Z^n \cap [-T,T]^n : [ \tilde{\M}_{\mathcal{B},r}  F_{E,r+T}(\omega, \cdot)](m) > \alpha\}d\mu(\omega)
\\
&\leq \frac{1}{(2T + 1)^n} \int_{\Omega} \#\{ m \in \Z^n:   \tilde{\M}_{\mathcal{B}}  \chi_{E_{r,T,\omega}} (m) > \alpha\}d\mu(\omega)
\end{split}
\]
where $E_{r,T,\omega}$ is the set
\[
E_{r,T,\omega}\coloneqq\{t\in\N^n\cap (-r-T,r+T)^n: U_1 ^{t_1}\cdots U_n ^{t_n}\omega \in E \}.
\]
By the definition of the sharp Tauberian constant $ \tilde{\mathsf C} 	 _{\mathcal{B}}$ we can thus conclude that
\[
\begin{split}
\mu(\{\omega \in \Omega : \M^{\ast}_{\mathcal{B},r} \chi_E(\omega) > \alpha\})& \leq \frac{1}{(2T + 1)^n}\int_\Omega    \tilde{\mathsf C} _{\mathcal{B}}(\alpha)  \# E_{R,T,\omega} d\mu(\omega)
\\
& \leq \frac{(2T + 2r + 1)^n}{(2T + 1)^n}  \tilde{ \mathsf  C} _{\mathcal{B}}(\alpha) \mu(E)
\end{split}
\]
by another application of Fubini's theorem. Letting $T$ tend to infinity yields
\[
\mu(\{\omega \in \Omega :  \M^{\ast}_{\mathcal{B}, r} \chi_E(\omega) > \alpha\}) \leq  \tilde{\mathsf C }_{\mathcal{B}}(\alpha) \mu(E).
\]
Subsequently letting $r$ tend to infinity yields
\[
\mu(\{\omega \in \Omega : \M^\ast_{\mathcal{B}} \chi_E(\omega) > \alpha\}) \leq  \tilde{\mathsf C}_{\mathcal{B}}(\alpha) \mu(E).
\]
Hence
\[
\mathsf C^{\ast}_{\mathcal{B}}(\alpha) \leq \tilde{\mathsf C}_{\mathcal{B}}(\alpha)
\]
and the proof is complete.
\end{proof}

\section{Solyanik estimates for the ergodic strong maximal operator}
In this section we establish the Solyanik estimate for the ergodic strong maximal operator. The proof that follows utilizes the transference principle developed in the previous section together with the Euclidean counterpart of the desired estimate, contained in Theorem~\ref{strongsolyanik}.
\begin{proof}[Proof of Theorem \ref{t1}]

We now establish the desired Solyanik estimates for the ergodic strong maximal operator $\Ms$ associated with some collection of commuting measure preserving transformations $U_1, \ldots, U_n$ on a probability space $(\Omega, \Sigma, \mu)$. Remember that $\Ms=\M^\ast _{\mathcal B_{\mathsf S}}$ where $\mathcal B _{\mathsf S}$ denotes the collection of open rectangles in $\R^n$ that contain the origin and have sides parallel to the coordinate axes. By the transference theorem of the preceding section, we realize that it suffices to prove an equivalent Solyanik estimate for the corresponding discrete strong maximal operator $\DSM$, acting on functions on $\mathbb{Z}^n$.  We will obtain such Solyanik estimates by taking advantage of known Solyanik estimates for the geometric strong maximal operator $\Mstr$ acting on functions on $\mathbb{R}^n$, given by Theorem \ref{strongsolyanik}.

To that end, let us recall here that $\DSM$ is given by
\[
\DSM f(m) \coloneqq  \sup_{0 \in R \in \mathbb{R}^n}\frac{1}{\#(R \cap \mathbb{Z}^n)}\sum_{ j = (j_1, \ldots, j_n)\in R \cap \mathbb{Z}^n}  |f(m+j)|,\quad m\in \Z^n,
\]
where the supremum is taken over all open rectangles in $\mathbb{R}^n$ containing the origin whose sides are parallel to the coordinate axes.

To each set $ \tilde{E}  \subset \mathbb{Z}^n$ we associate a set $ E \subset \mathbb{R}^n$ defined by
\[
\chi_{ E}(x_1, \ldots, x_n) \coloneqq \chi_{ \tilde{E} }(\lfloor x_1 \rfloor, \ldots, \lfloor x_n \rfloor),\quad  (x_1,\ldots,x_n)\in \R^n.
\]
Here, for $z\in \R$  we denote by $\lfloor z \rfloor$ the greatest integer which is less or equal than $z$. With this definition we have
\[
\# \tilde{E} = \sum_{(j_1,\ldots,j_n)\in \tilde{E}} \ABs{\prod_{k=1} ^n [j_k,j_k+1)}=| E|.
\]
Let $m=(m_1,\ldots,m_n)\in \Z^n$ and $R\in \mathcal B_{\mathsf S}$. For any set $ \tilde{E} \subseteq \Z^n$ we have that
\[
\begin{split}
\frac{1}{\#(R\cap\Z^n)}\sum_{j\in R} \chi_{ \tilde{E}}(m+j)& =\frac{1}{\#(R\cap\Z^n)}\sum_{j\in R} \int_{\prod_{k=1} ^n[m_k+j_k,m_k+j_k+1)} \chi_{E}(\lfloor u \rfloor )du
\\
& =  \frac{1}{| R'_m |}\int_{ R'_m } \chi_{E}(u)du
\end{split}
\]
where $ R'_m \subseteq \R^n$ is a rectangle in $\R^n$ whose sides are parallel to the axes with $R' _m\supseteq R_m := (m_1,m_1+1)\times \cdots\times (m_n,m_n+1) $. Note then $\inf_{ 	R_m} \Mstr\chi_{E} \geq\DSM \chi_E(m)$. Defining \mbox{$ \tilde{E}_\alpha\coloneqq \{m\in\Z^n:\DSM \chi_{ \tilde{E} }(m) > \alpha\}$} we thus have
\[
\begin{split}
\# \tilde{E}_\alpha &=\Abs{\bigcup_{m\in \tilde{E}_\alpha}  R_m}\leq |\{x\in\R^n:\Mstr\chi_{E}(x)>\alpha\}|
\\
&\leq \Cstr(\alpha)| E|=\Cstr(\alpha)\# \tilde E
\end{split}
\]
by the definition of sharp Tauberian constant $\Cstr(\alpha)$ and the fact that $\#\tilde{E}=|E|.$ Thus, recalling that $\tilde{\mathsf{C}}_{\mathsf S}(\alpha)$ denotes the sharp Tauberian constant with respect to $\alpha$ associated to $\tilde{\M}_{\mathsf S}$, we have proven that $\tilde{\mathsf{C}}_{\mathsf S}(\alpha) \leq \Cstr(\alpha)$ for all $\alpha\in(0,1)$. Now the transference result of Theorem~\ref{tt} together with the Solyanik estimate for the strong maximal function of Theorem~\ref{t1} imply that
\[
\Cs(\alpha)-1\leq \tilde{\mathsf{C}}_{S}(\alpha) - 1 \lesssim_n \big(\frac{1}{\alpha} - 1\big)^{\frac{1}{n}}
\]
for $\alpha$ sufficiently close to $1$, completing the proof of the theorem.
\end{proof}

\section{Solyanik estimates for the centered ergodic Hardy-Littlewood maximal operator}

In this section we give the proof of the Solyanik estimate for the centered ergodic Hardy-Littlewood maximal operator. The proof relies, again, on the transference principle of Theorem~\ref{tt}.

\begin{proof}[Proof of Theorem \ref{t2}] Recall that $\MHLCD$ is the discrete centered Hardy-Littlewood maximal operator defined on $L^1(\mathbb{Z}^n)$ by
\[
\MHLCD f(m) \coloneqq \sup_{B\in\mathcal B_{\mathsf{HL,c}}}\frac{1}{\#(B\cap\Z^n)}\sum_{j=(j_1, \ldots, j_n)\in B\cap \mathbb{Z}^n}|f(m+j)|,\quad m \in \Z^n.
\]
We will show that $\MHLCD$ satisfies the Solyanik estimate
\[
\CHLCD(\alpha) - 1 \lesssim_n \frac{1}{\alpha} - 1
\]
for $\alpha$ sufficiently close to $1$. Recall here that $\CHLCD(\alpha) $ is the sharp Tauberian constant of $\MHLCD$ associated with $\alpha$ as defined in \S\ref{intro}.

Fix now $0 < \alpha < 1$, and let $\tilde E $ be a nonempty finite subset of $\mathbb{Z}^n$. Setting
$ \tilde E _\alpha \coloneqq\{m\in \Z^n: \MHLCD \chi_{\tilde E }(m)>\alpha\}$ it is easy to see that $\tilde E _\alpha$ is a finite set. Then, for each $m\in  \tilde  E _\alpha$ there exists a Euclidean ball $ B^m \in \R^n$ such that
\[
\frac{1}{\#(B^m \cap \Z^n)}\sum_{w \in  B^m\cap \mathbb{Z}^n }\chi_{\tilde E }(w)  > \alpha \quad\text{and}\quad \tilde E_\alpha\subseteq \bigcup_{m\in \tilde  E _\alpha} B^m.
\]
By the Besicovitch covering theorem, as for example in \cite{mattila}, there exists a subcollection $\{B_j\}_{j=1} ^N \subseteq \{B^m\}_{m\in \tilde E_\alpha}$ such that $\tilde E _\alpha \subseteq \cup_j   B_j \cap \Z^n $ and $\sum_j \chi_{B_j} \leq A_n$, where $A_n>0$ is a dimensional constant. We can now estimate
\[
\begin{split}
\#  \tilde E _\alpha & \leq  \# \tilde E +\# \bigcup_j (B_j\cap \Z^n )\setminus \tilde E \leq \# \tilde E  + \sum_j \# (B_j\cap \Z^n )\setminus   \tilde E
\\
&\leq \# \tilde E  + \frac{1-\alpha}{\alpha} \sum_j \# B_j \cap \tilde  E  \leq \#  \tilde E  +A_n \frac{1-\alpha}{\alpha} \#  \tilde E .
\end{split}
\]
This shows that $\CHLCD(\alpha)\leq 1+A_n\frac{1-\alpha}{\alpha} $. Now Theorem~\ref{tt} implies that
\[
\CHLCE(\alpha)-1 \leq  \CHLCD(\alpha) -1 \lesssim_n \frac{1-\alpha}{\alpha}
\]
as we wanted to show.
\end{proof}

\section{Solyanik estimates for the uncentered Hardy-Littlewood ergodic maximal operator}

In this section we show the Solyanik estimate for the uncentered ergodic Hardy-Littlewood maximal operator. The proof follows again the familiar pattern of proving a corresponding result for a suitable discrete geometric maximal operator and then using the transference principle of \S\ref{s.trans}.

\begin{proof}[Proof of Theorem \ref{t3}]  Let us consider the discrete uncentered maximal operator $\MHLD$ defined on $L^1(\mathbb{Z}^n)$ by
\[
\MHLD f(m) = \sup_{B\in \mathcal B_{\mathsf{HL}}} \frac{1}{\#(B\cap\Z^n)} \sum_{j =(j_1, \ldots, j_n)\in B \cap \mathbb{Z}^n}  |f(m+j)|,\quad m\in \Z^n.
\]
We will show that $\MHLD $ satisfies the Solyanik estimate
\[
\mathsf \CHLD(\alpha) - 1 \lesssim_n \left(\frac{1}{\alpha} - 1\right)^{\frac{1}{n(n+1)}}
\]
for $\alpha$ sufficiently close to $1$, where $\CHLD(\alpha)$ is the sharp Tauberian constant of $\MHLD$ associated with $\alpha$, as defined in \S\ref{intro}.

Fix now $0<\alpha<1$, and let $\tilde{E}$ be a nonempty finite subset of $\mathbb{Z}^n$.  Set $ \tilde{E}_\alpha \coloneqq\{m\in \Z^n: \MHLD \chi_{\tilde{E}}(m)>\alpha\}$. We may assume without loss of generality that $\tilde E_\alpha\setminus \tilde E\neq \emptyset$.

Suppose  that $m\in \tilde{E}_\alpha\setminus \tilde{E} $. Then there exists a Euclidean ball $B_m\subset \R^n$ such that $m\in B_m$ and
\[
\frac{1}{\#(B_m\cap \Z^n)} \sum_{w\in B_m\cap \Z^n} \chi_{\tilde{E}}(w)>\alpha.
\]
Furthermore, since $m\in B_m \cap \Z^n \setminus \tilde{E}\neq \emptyset  $  we have the elementary estimate
\[
\alpha<\frac{\#(B_m\cap \Z^n)\cap \tilde{E}}{\#(B_m\cap \Z^n)}\leq \frac{\#(B_m\cap \Z^n)-1}{\#(B_m\cap \Z^n)}
\]
and thus $\#(B_m\cap \Z^n)>  (1-\alpha)^{-1}$. Letting $c_m$ denote the center of $B_m$ and $r_m$ denote the radius of $B_m$, elementary geometric considerations imply that
\[
\bigcup_{w \in B_m\cap\Z^n} (w+[-1,1)^n ) \subset B(c_m, r_m + \sqrt{n}),
\]
where we let $B(c,r)$ denote the open ball in $\mathbb{R}^n$ of radius $r$ centered at $c$.  Moreover
\[
B(c_m, r_m - \sqrt{n}) \subset \bigcup_{w \in B_m\cap \Z^n } (w+[0,1)^n  ).
\]
So
\[
C_n (r_m - \sqrt{n})^n \leq \#(B_m\cap \Z^n )\leq C_n(r_m + \sqrt{n})^n
\]
for some dimensional constant $C_n>0$.

As we have done previously, associate now to the discrete set $ \tilde{E} \subset\Z^n$ the set $E \subset\R^n$ defined by
\[
\chi_{E}(t_1, \ldots, t_n) \coloneqq \chi_{ \tilde{E} }(\lfloor t_1 \rfloor, \ldots, \lfloor t_n \rfloor),\quad (t_1,\ldots,t_n)\in\R^n.
\]
Observe that $$B(m, \sqrt{n}) \subset B(c_m, r_m + \sqrt{n})$$ so that for all $y \in B(m,\sqrt{n})$ we have
\[
\begin{split}
\MHL \chi_{E}(y) & \geq  \frac{\# ( \tilde{E}  \cap B_m)}{C_n(r_m + \sqrt{n})^n} > \alpha \frac{\#(B_m\cap \Z^n)}{C_n(r_m + \sqrt{n})^n}
\\
& > \alpha\frac{C_n(r_m - \sqrt{n})^n}{C_n(r_m + \sqrt{n})^n} = \alpha\left(\frac{r_m - \sqrt{n}}{r_m + \sqrt{n}}\right)^n .
\end{split}
\]
Note that  $C_n(r_m + \sqrt{n})^n \geq \#(B_m\cap \Z^n) >(1 - \alpha)^{-1}$ and thus $r_m\geq (C_n(1-\alpha))^{- \frac{1}{n}}-\sqrt n$. Thus the previous estimate implies that
\[
\MHL \chi_{E}(y) > \alpha\bigg(\frac{(C_n(1-\alpha))^{- \frac{1}{n}}-2\sqrt n}{(C_n(1-\alpha))^{- \frac{1}{n}} + \sqrt{n}}\bigg)^n\eqqcolon c(\alpha,n),\qquad  \forall y\in B(m,\sqrt n).
\]
We conclude that
\[
m\in \tilde{E}_\alpha\setminus \tilde{E} \Rightarrow \MHL\chi_{E}>c(\alpha,n)\quad\text{on}\quad (m_1,m_1+1)\times \cdots\times (m_n,m_n+1)\coloneqq Q_m.
\]

Suppose on the other hand that $m\in \tilde{E} \cap \tilde{E}_{\alpha}$. Then $Q_m\subseteq E$ and thus $\MHL \chi_E$ is identically $1$ on $Q_m$.

Combining the estimates and observations above we see that if $m \in \tilde{E}_{\alpha}$ we must have $\MHL \chi_{E}(x)>c(\alpha,n)$ on $Q_m$.  Hence
\[
\begin{split}
\#\tilde{E}_\alpha &=\ABs{\bigcup_{m\in \tilde{E}_\alpha} Q_m } \leq |\{x\in\R^n: \MHL \chi_{E}(x)>c(\alpha,n)\}|	
\\
& \leq \CHL(c(\alpha,n))| E|=\CHL(c(\alpha,n)) \#\tilde{E}.
\end{split}
\]
This proves the estimate
\[
\CHLD(\alpha)\leq \CHL(c(\alpha,n)).	
\]
It is obvious that $\lim_{\alpha\to 1^-}c(\alpha,n)=1$. Thus, if $\alpha$ is sufficiently close to $1$, depending only upon the dimension $n$, Theorem~\ref{solyanikthm} implies
\[
\CHLD(\alpha) -1\lesssim_n (1-c(\alpha,n))^\frac{1}{n+1}.
\]
By direct computation one may show that
\[
(1-c(\alpha,n))^\frac{1}{n+1}\lesssim (1-\alpha)^\frac{1}{n(n+1)}\quad\text{as}\quad  \alpha\to 1^-.
\]
Thus the previous estimate together with the transference principle of Theorem~\ref{tt} show that
\[
\CHLE(\alpha)-1\leq \CHLD(\alpha) -1 \lesssim_n (1-c(\alpha,n))^\frac{1}{n+1}\lesssim_n (1-\alpha)^\frac{1}{n(n+1)}
\]
for $\alpha$ sufficiently close to $1$, as desired.
\end{proof}

\section{Future directions for research involving Solyanik estimates in ergodic theory}

Our original foray into the topic of Solyanik estimates in ergodic theory has been promising, and we close here with three problems that we believe to be appropriate directions for future development in the subject.

\begin{problem} An intriguing question is to whether the exponent $\frac{1}{n(n+1)}$ occurring in the Solyanik estimate of Theorem \ref{t3} is sharp, in particular holding for all choices of commuting measure preserving transformations $U_1, \ldots, U_n$.  Indeed we do not know if the exponent $\frac{1}{n+1}$ is sharp for the Solyanik exponent associated to $\CHL(\alpha)$ given in Theorem \ref{solyanikthm}.  The problem for the optimal Solyanik exponent is discussed in detail in \cite{HP}, where evidence is given that suggests the optimal exponent might be as large as $\frac{1}{n}$ or possibly even $\frac{2}{n+1}$.
\end{problem}

\begin{problem} It is natural to ask, provided $\mathcal{B}$ is any sort of reasonable basis, whether or not ergodic Solyanik estimates must hold for $\CEB$.  In particular, if $\mathcal{B}$ is a basis consisting of convex subsets in $\mathbb{Z}^n$  such that $\CEB < \infty$ for every $0 < \alpha < 1$, must \[\lim_{\alpha \rightarrow 1^-}\CEB = 1\] hold?  It is highly unlikely that the convexity condition can be dispensed with; see \cite{bh} for comments regarding bases of nonconvex sets in $\mathbb{R}^n$ for which Solyanik estimates for the associated geometric maximal operators are known not to hold.
\end{problem}

\begin{problem} It is not hard to see that for certain choices of commuting measure preserving transformations $U_1,\ldots,U_n$ on a probability space $(\Omega, \Sigma, \mu)$ one might obtain especially good Solyanik estimates for $\CEB$. Indeed, consider for example the case $U_1=\cdots=U_n=\mathsf{Id}$ where $\mathsf{Id}$ is the identity operator. It is natural to consider collections of transformations $U_1, \ldots, U_n$  which yield the worst possible Solyanik exponent associated to a given basis.  We suspect that, in many cases,  the worst possible exponent may be obtained by requiring that $U_1, \ldots, U_n$ be \emph{non-periodic} in the sense of Katznelson and Weiss; see \cite{KW}.  It would be natural to first test this hypothesis in the special case that $\mathcal{B}$ is the relatively well-understood basis $\mathcal{B}_{\mathsf{HL}}$ or $\mathcal{B}_{\mathsf{S}}$, and very likely the techniques devised by Hagelstein and Stokolos in \cite{hs2012} on sharp transference estimates would be helpful here.
\end{problem}

\begin{problem}  The authors have recently shown in \cite{HP2014Holder} that Solyanik estimates may be used to establish \emph{smoothness} results for the functions $\CHL (\alpha)$ and $\Cstr (\alpha)$ on $(0,1)$.   In particular, they proved that both lie in the H\"older class $C^{\frac{1}{n}}(0,1)$.  It is natural to consider whether or not, in the ergodic setting, the Tauberian constants $\CHLE (\alpha)$ and $\Cs(\alpha)$ satisfy H\"older continuity estimates or are possibly even differentiable or smooth.
\end{problem}

\end{document}